\documentclass[]{article}
\usepackage{amssymb}
\usepackage{amsmath}
\usepackage{amsfonts}

\newtheorem{theorem}{Theorem}
\newtheorem{example}[theorem]{Example}
\newtheorem{proposition}[theorem]{Proposition}
\newtheorem{corollary}[theorem]{Corollary}
\newtheorem{proof}[theorem]{Proof}

\begin{document}

\title{ON GENERALIZED SPHERICAL SURFACES IN EUCLIDEAN SPACES}
\author{Beng\"{u} K\i l\i \c{c} Bayram, Kadri Arslan \& Bet\"{u}l Bulca}
\maketitle

\begin{abstract}
In the present study we consider the generalized rotational surfaces in
Euclidean spaces. Firstly, we consider generalized spherical curves in
Euclidean $(n+1)-$space $\mathbb{E}^{n+1}$. Further, we introduce some kind
of generalized spherical surfaces in Euclidean spaces $\mathbb{E}^{3}$ and $%
\mathbb{E}^{4}$ respectively. We have shown that the generalized spherical
surfaces of first kind in $\mathbb{E}^{4}$ are known as rotational surfaces,
and the second kind generalized spherical surfaces are known as meridian
surfaces in $\mathbb{E}^{4}$. We have also calculated the Gaussian, normal
and mean curvatures of these kind of surfaces. Finally, we give some
examples.
\end{abstract}

\section{\textbf{Introduction}}

\footnote{%
2010 AMS \textit{Mathematics Subject Classification}. 53C40, 53C42
\par
\textit{Key words and phrases}: Second fundamental form, Gaussian curvature,
rotational surface, Otsuiki surface}The Gaussian curvature and mean
curvature of the surfaces in Euclidean spaces play an important role in
differential geometry. Especially, surfaces with constant Gaussian curvature
\cite{Wo}, and constant mean curvature conform nice classes of surfaces
which are important for surface modelling \cite{BABO2}. Surfaces with
constant negative curvature are known as pseudo-spherical surfaces \cite{GN}.

Rotational surfaces in Euclidean spaces are also important subject of
differential geometry. The rotational surfaces in $\mathbb{E}^{3}$ are
called surface of revolution. Recently V. Velickovic classified all
rotational surfaces in $\mathbb{E}^{3}$ with constant Gaussian curvature 
\cite{Ve}. Rotational surfaces in $\mathbb{E}^{4}$ was first introduced by
C. Moore in $1919$. In the recent years some mathematicians have taken an
interest in the rotational surfaces in $\mathbb{E}^{4}$, for example G.
Ganchev and V. Milousheva \cite{GM1}, U. Dursun and N. C. Turgay \cite{DT},
K. Arslan, at all. \cite{ABCO} and D.W.Yoon  \cite{Yo}. In \cite{GM1}, the
authors applied invariance theory of surfaces in the four dimensional
Euclidean space to the class of general rotational surfaces whose meridians
lie in two-dimensional planes in order to find all minimal super-conformal
surfaces. These surfaces were further studied in \cite{DT}, which found all
minimal surfaces by solving the differential equation that characterizes
minimal surfaces. They then determined all pseudo-umbilical general
rotational surfaces in $\mathbb{E}^{4}$. K. Arslan et.al in \cite{ABCO} gave
the necessary and sufficient conditions for generalized rotation surfaces to
become pseudo-umbilical, they also shown that each general rotational
surface is a Chen surface in $\mathbb{E}^{4}$ and gave some special classes
of generalized rotational surfaces as examples. See also \cite{Cu} and \cite%
{BABO1} rotational surfaces with Constant Gaussian Curvature in Four-Space.
For higher dimensional case N.H. Kuiper defined rotational embedded
submanifolds in Euclidean spaces \cite{Ku}.

The meridian surfaces in $\mathbb{E}^{4}$ was first introduced by G. Ganchev
and V. Milousheva (See, \cite{GM2} and \cite{ABM}) which are the special
kind of rotational surfaces. Basic source of examples of surfaces in
4-dimensional Euclidean or pseudo-Euclidean space are the standard
rotational surfaces and the general rotational surfaces. Further, Ganchev
and Milousheva defined another class of surfaces of rotational type which
are one-parameter system of meridians of a rotational hypersurface. They
constructed a family of surfaces with at normal connection lying on a
standard rotational hypersurface in $\mathbb{E}^{4}$ as a meridian surfaces.
The geometric construction of the meridian surfaces is different from the
construction of the standard rotational surfaces with two dimensional axis
in $\mathbb{E}^{4}$.

This paper is organized as follows: Section $2$ gives some basic concepts of
the surfaces in $\mathbb{E}^{n}$. Section $3$ explains some geometric
properties of spherical curves $\mathbb{E}^{n+1}$. Section $4$ tells about
the generalized spherical surfaces in $\mathbb{E}^{n+m}$. Further this
section provides some basic properties of generalized spherical surfaces in $%
\mathbb{E}^{4}$ and the structure of their curvatures. We also shown that
every generalized spherical surfaces in $\mathbb{E}^{4}$ have constant
Gaussian curvature $K=1/c^{2}$. Finally, we present some examples of
generalized spherical surfaces in $\mathbb{E}^{4}$.

\section{\textbf{Basic Concepts}}

Let $M$ be a smooth surface in $\mathbb{E}^{n}$ given with the patch $X(u,v)$
: $(u,v)\in D\subset \mathbb{E}^{2}$. The tangent space to $M$ at an
arbitrary point $p=X(u,v)$ of $M$ span $\left \{ X_{u},X_{v}\right \} $. In
the chart $(u,v)$ the coefficients of the first fundamental form of $M$ are
given by 
\begin{equation}
g_{11}=\left \langle X_{u},X_{u}\right \rangle ,g_{12}=\left \langle
X_{u},X_{v}\right \rangle ,g_{22}=\left \langle X_{v},X_{v}\right \rangle ,
\label{A1}
\end{equation}%
where $\left \langle ,\right \rangle $ is the Euclidean inner product. We
assume that $W^{2}=g_{11}g_{22}-g_{12}^{2}\neq 0,$ i.e. the surface patch $%
X(u,v)$ is regular.\ For each $p\in M$, consider the decomposition $T_{p}%
\mathbb{E}^{n}=T_{p}M\oplus T_{p}^{\perp }M$ where $T_{p}^{\perp }M$ is the
orthogonal component of $T_{p}M$ in $\mathbb{E}^{n}$.

Let $\chi (M)$ and $\chi ^{\perp }(M)$ be the space of the smooth vector
fields tangent to $M$ and the space of the smooth vector fields normal to $M$%
, respectively. Given any local vector fields $X_{1},$ $X_{2}$ tangent to $M$%
, consider the second fundamental map $h:\chi (M)\times \chi (M)\rightarrow
\chi ^{\perp }(M);$%
\begin{equation}
h(X_{i},X_{_{j}})=\widetilde{\nabla }_{X_{_{i}}}X_{_{j}}-\nabla
_{X_{_{i}}}X_{_{j}}\text{ \  \  \ }1\leq i,j\leq 2  \label{A2}
\end{equation}%
where $\nabla $ and $\overset{\sim }{\nabla }$ are the induced connection of 
$M$ and the Riemannian connection of $\mathbb{E}^{n}$, respectively. This
map is well-defined, symmetric and bilinear \cite{Ch1}.

For any arbitrary orthonormal frame field $\left \{
N_{1},N_{2},...,N_{n-2}\right \} $ of $M$, recall the shape operator $A:\chi
^{\perp }(M)\times \chi (M)\rightarrow \chi (M);$%
\begin{equation}
A_{N_{k}}X_{j}=-(\widetilde{\nabla }_{X_{j}}N_{k})^{T},\text{ \  \  \ }%
X_{j}\in \chi (M).  \label{A3}
\end{equation}%
This operator is bilinear, self-adjoint and satisfies the following equation:%
\begin{equation}
\left \langle A_{N_{k}}X_{j},X_{i}\right \rangle =\left \langle
h(X_{i},X_{j}),N_{k}\right \rangle =L_{ij}^{k}\text{, }1\leq i,j\leq 2;\text{
}1\leq k\leq n-2  \label{A4}
\end{equation}%
where $L_{ij}^{k}$ are the coefficients of the second fundamental form. The
equation (\ref{A2}) is called Gaussian formula, and%
\begin{equation}
h(X_{i},X_{j})=\overset{n-2}{\underset{k=1}{\sum }}L_{ij}^{k}N_{k},\  \  \  \  \
1\leq i,j\leq 2  \label{A5}
\end{equation}%
holds. Then the Gauss curvature $K$ of a regular patch $X(u,v)$ is given by

\begin{equation}
K=\frac{1}{W^{2}}\sum%
\limits_{k=1}^{n-2}(L_{11}^{k}L_{22}^{k}-(L_{12}^{k})^{2}).  \label{A6}
\end{equation}

Further, the mean curvature vector of a regular patch $X(u,v)$ is given by 
\begin{equation}
\overrightarrow{H}=\frac{1}{2W^{2}}%
\sum_{k=1}^{n-2}(L_{11}^{k}g_{22}+L_{22}^{k}g_{11}-2L_{12}^{k}g_{12})N_{k}.
\label{A7}
\end{equation}%
We call the functions 
\begin{equation}
H_{k}=\frac{(L_{11}^{k}g_{22}+L_{22}^{k}g_{11}-2L_{12}^{k}g_{12})}{2W^{2}},
\label{A8}
\end{equation}%
the $k.th$ mean curvature functions of the given surface. The norm of the
mean curvature vector $H=\left \Vert \overrightarrow{H}\right \Vert $ is
called the mean curvature of $M$. Recall that a surface $M$ is said to be 
\textit{flat} (resp. \textit{minimal})\textit{\ } if its Gauss curvature
(resp. mean curvature vector) vanishes identically \cite{Ch2}, \cite{Ch3}.

The normal curvature $K_{N}$ of $M$ is defined by (see \cite{DDVV}) 
\begin{equation}
K_{N}=\left \{ \overset{n-2}{\underset{1=\alpha <\beta }{\sum }}\left
\langle R^{\bot }(X_{1},X_{2})N_{\alpha },N_{\beta }\right \rangle
^{2}\right \} ^{1/2}.  \label{A9}
\end{equation}%
where%
\begin{equation}
R^{\bot }(X_{i},X_{j})N_{\alpha }=h(X_{i},A_{N_{\alpha
}}X_{j})-h(X_{j},A_{N_{\alpha }}X_{i}),  \label{A10}
\end{equation}%
and 
\begin{equation}
\text{ \  \  \  \  \ }\left \langle R^{\bot }(X_{i},X_{j})N_{\alpha },N_{\beta
}\right \rangle =\left \langle [A_{N_{\alpha }},A_{N_{\beta
}}]X_{i,}X_{j}\right \rangle ,  \label{A11}
\end{equation}%
is called the \textit{equation of Ricci}. We observe that the normal
connection $D$ of $M$ is flat if and only if $K_{N}=0,$ and by a result of
Cartan, this equivalent to the diagonalisability of all shape operators $%
A_{N_{\alpha }}$ \cite{Ch1}.\ 

\section{\textbf{Generalized Spherical Curves}}

Let $\gamma $ be a regular oriented curve in $\mathbb{E}^{n+1}$ that does
not lie in any subspace of $\mathbb{E}^{n+1}$. From each point of the curve $%
\gamma $ one can draw a segment of unit length along the normal line
corresponding to the chosen orientation. The ends of these segments describe
a new curve $\beta$. The curve $\gamma \in \mathbb{E}^{n+1}$ is called a 
\textit{generalized spherical curve} if the curve $\beta $ lies in a certain
subspace $\mathbb{E}^{n}$of $\mathbb{E}^{n+1}.$ The curve $\beta $ is called
the trace of $\gamma $ \cite{GN}$.$ Let 
\begin{equation}
\gamma (u)=\left( f_{1}(u),...,f_{n+1}(u)\right) ,  \label{B1}
\end{equation}%
be the radius vector of the curve $\gamma $ given with arclangth
parametrization $u$, i.e., $\left \Vert \gamma ^{\prime }(u)\right \Vert =1.$
The curve $\beta $ is defined by the radius vector%
\begin{equation}
\beta (u)=(\gamma +c^{2}\gamma ^{\prime \prime })(u)=\left(
(f_{1}+c^{2}f_{1}^{\prime \prime })(u),...,(f_{n+1}+c^{2}f_{n+1}^{\prime
\prime })(u)\right) ,  \label{B2}
\end{equation}%
where $c$ is a real constant. If $\gamma $ is a generalized spherical curve
of $\mathbb{E}^{n+1}$ then by definition the curve $\beta $ lies in the
hyperplane $\mathbb{E}^{n}$ if and only if $f_{n+1}+c^{2}f_{n+1}^{\prime
\prime }=0.$ Consequently, this equation has a non-trivial solution%
\begin{equation*}
f_{n+1}(u)=\lambda \cos \left( \frac{u}{c}+c_{0}\right) ,
\end{equation*}%
with some constants $\lambda $ and $c_{0}.$ By a suitable choose of
arclenght we may assume that 
\begin{equation}
f_{n+1}(u)=\lambda \cos \left( \frac{u}{c}\right) ,  \label{B3}
\end{equation}%
with $\lambda >0.$ Thus, the radius vector of the generalized spherical
curve $\gamma $ takes the form%
\begin{equation}
\gamma (u)=\left( f_{1}(u),...,f_{n}(u),\lambda \cos \left( \frac{u}{c}%
\right) \right) .  \label{B4}
\end{equation}%
Moreover, the condition for the arclength parameter $u$ implies that%
\begin{equation}
(f_{1}^{\prime })^{2}+...+(f_{n}^{\prime })^{2}=1-\frac{\lambda ^{2}}{c^{2}}%
\sin ^{2}\left( \frac{u}{c}\right) .  \label{B5}
\end{equation}%
For convenience, we introduce a vector function%
\begin{equation*}
\phi (u)=\left( f_{1}(u),...,f_{n}(u);0\right) .
\end{equation*}%
Then the radius vector (\ref{B4}) can be represented in the form%
\begin{equation}
\gamma (u)=\phi (u)+\lambda \cos \left( \frac{u}{c}\right) e_{n+1},
\label{B6}
\end{equation}%
where $e_{n+1}=(0,0,...,0,1).$ Consequently, the condition (\ref{B5}) gives%
\begin{equation}
\left \Vert \phi ^{\prime }(u)\right \Vert ^{2}=1-\frac{\lambda ^{2}}{c^{2}}%
\sin ^{2}\left( \frac{u}{c}\right) .  \label{B7}
\end{equation}%
Hence, the radius vector of the trace curve $\beta $ becomes%
\begin{equation}
\beta (u)=\phi (u)+c^{2}\phi ^{\prime \prime }(u).  \label{B8}
\end{equation}

Consider an arbitrary unit vector function%
\begin{equation}
a(u)=\left( a_{1}(u),...,a_{n}(u);0\right) ,  \label{B9}
\end{equation}%
in $\mathbb{E}^{n+1}$ and use this function to construct a new vector
function%
\begin{equation}
\phi (u)=\int \sqrt{1-\frac{\lambda ^{2}}{c^{2}}\sin ^{2}\left( \frac{u}{c}%
\right) }a(u)du,  \label{B10}
\end{equation}%
whose last coordinate is equal to zero. Consequently, the vector function $%
\phi (u)$ satisfies the condition (\ref{B7}) and generates a generalized
spherical curve with radius vector (\ref{B6}).

\begin{example}
The ordinary circular curve in $\mathbb{E}^{2}$ is given with the radius
vector 
\begin{equation}
\gamma (u)=\left( \int \sqrt{1-\frac{\lambda ^{2}}{c^{2}}\sin ^{2}\left( 
\frac{u}{c}\right) }du,\lambda \cos \left( \frac{u}{c}\right) \right) .
\label{B11}
\end{equation}
\end{example}

\begin{example}
Consider the unit vector $a(u)=\left( \cos \alpha (u),\sin \alpha
(u);0\right) $ in $\mathbb{E}^{2}$ . Then using (\ref{B10}), the
corresponding generalized spherical curve in $\mathbb{E}^{3}$ is defined by
the radius vector%
\begin{eqnarray}
f_{1}(u) &=&\int \sqrt{1-\frac{\lambda ^{2}}{c^{2}}\sin ^{2}\left( \frac{u}{c%
}\right) }\cos \alpha (u)du,  \notag \\
f_{2}(u) &=&\int \sqrt{1-\frac{\lambda ^{2}}{c^{2}}\sin ^{2}\left( \frac{u}{c%
}\right) }\sin \alpha (u)du,  \label{B12} \\
f_{3}(u) &=&\lambda \cos \left( \frac{u}{c}\right) .  \notag
\end{eqnarray}
\end{example}

\begin{example}
Consider the unit vector%
\begin{equation*}
a(u)=\left( \cos \alpha (u),\cos \alpha (u)\sin \alpha (u),\sin ^{2}\alpha
(u);0\right)
\end{equation*}%
in $\mathbb{E}^{3}$ . Then using (\ref{B10}), the corresponding generalized
spherical curve in $\mathbb{E}^{4}$ is defined by the radius vector%
\begin{eqnarray}
f_{1}(u) &=&\int \sqrt{1-\frac{\lambda ^{2}}{c^{2}}\sin ^{2}\left( \frac{u}{c%
}\right) }\cos \alpha (u)du,  \notag \\
f_{2}(u) &=&\int \sqrt{1-\frac{\lambda ^{2}}{c^{2}}\sin ^{2}\left( \frac{u}{c%
}\right) }\cos \alpha (u)\sin \alpha (u)du,  \label{B13} \\
f_{3}(u) &=&\int \sqrt{1-\frac{\lambda ^{2}}{c^{2}}\sin ^{2}\left( \frac{u}{c%
}\right) }\sin ^{2}\alpha (u)du;  \notag \\
f_{4}(u) &=&\lambda \cos \left( \frac{u}{c}\right) .  \notag
\end{eqnarray}
\end{example}

\section{\textbf{Generalized Spherical Surfaces}}

Consider the space $\mathbb{E}^{n+1}=\mathbb{E}^{n}\oplus \mathbb{E}^{1}$ as
a subspace of $\mathbb{E}^{n+m}=\mathbb{E}^{n}\oplus \mathbb{E}^{m},$ $m\geq
2$ and Cartesian coordinates $x_{1},x_{2},...,x_{n+m}$ and orthonormal basis 
$e_{1},...,e_{n+m}$ in $\mathbb{E}^{n+m}$. Let $M^{2}$ be a local surface
given with the regular patch (radius vector) $\mathbb{E}^{n}\subset $ $%
\mathbb{E}^{n+1}$ 
\begin{equation}
X(u,v)=\phi (u)+\lambda \cos \left( \frac{u}{c}\right) \rho (v),  \label{C1}
\end{equation}%
where the vector function $\phi (u)=\left(
f_{1}(u),...,f_{n}(u),0,...,0\right) ,$ satisfies (\ref{B7}) and generates a
generalized spherical curve with radius vector

\begin{equation}
\gamma (u)=\phi (u)+\lambda \cos \left( \frac{u}{c}\right) e_{n+1},
\label{C2}
\end{equation}%
and the vector function $\rho (v)=\left(
0,...,0,g_{1}(v),...,g_{m}(v)\right) ,$ satisfying the conditions $%
\left
\Vert \rho (v)\right \Vert =1,\left \Vert \rho ^{\prime
}(v)\right
\Vert =1,$ and specifies a curve $\rho =\rho (v)$ parametrized by
a natural parameter on the unit sphere $S^{m-1}\subset $ $\mathbb{E}^{m}.$
Consequently, the surface $M$ $^{2}$ is obtained as a result of the rotation
of the generalized spherical curve $\gamma $ along the spherical curve $\rho
,$ which is called \textit{generalized Spherical surface} in $\mathbb{E}%
^{n+m}$.

In the sequel, we will consider some type of generalized spherical surface;

\bigskip

\textbf{CASE }$\mathbf{I}$\textbf{.} For $n=1$ and $m=2$, the radius vector (%
\ref{C1}) satisfying the indicated properties describes the \textit{%
spherical surface} in $\mathbb{E}^{3}$ with the radius vector%
\begin{equation}
X(u,v)=(\phi (u),\lambda \cos \left( \frac{u}{c}\right) \cos v,\lambda \cos
\left( \frac{u}{c}\right) \sin v),  \label{C3}
\end{equation}%
where the function $\phi (u)$ is found from the relation $\left \vert \phi
^{\prime }(u)\right \vert =\sqrt{1-\frac{\lambda ^{2}}{c^{2}}\sin ^{2}\left( 
\frac{u}{c}\right) }$. The surface given with the parametrization (\ref{C3})
is a kind of surface of revolution which is called ordinary sphere.

The tangent space is spanned by the vector fields%
\begin{eqnarray*}
X_{u}(u,v) &=&(\phi ^{\prime }(u),\frac{-\lambda }{c}\sin \left( \frac{u}{c}%
\right) \cos v,\frac{-\lambda }{c}\sin \left( \frac{u}{c}\right) \sin v), \\
X_{v}(u,v) &=&(0,-\lambda \cos \left( \frac{u}{c}\right) \sin v,\lambda \cos
\left( \frac{u}{c}\right) \cos (v)).
\end{eqnarray*}%
\ 

Hence, the coefficients of the first fundamental form of the surface are%
\begin{eqnarray*}
g_{11} &=&\text{ }<X_{u}(u,v),X_{u}(u,v)>\text{ }=1 \\
g_{12} &=&\text{ }<X_{u}(u,v),X_{v}(u,v)>\text{ }=0 \\
g_{22} &=&\text{ }<X_{v}(u,v),X_{v}(u,v)>\text{ }=\lambda ^{2}\cos
^{2}\left( \frac{u}{c}\right) ,
\end{eqnarray*}%
where $\left \langle ,\right \rangle $ is the standard scalar product in $%
\mathbb{E}^{3}$.

For a regular patch $X(u,v)$ the unit normal vector field or surface normal $%
N$ is defined by 
\begin{eqnarray*}
N(u,v) &=&\frac{X_{u}\times X_{v}}{\parallel X_{u}\times X_{v}\parallel }%
(u,v) \\
&=&\left( -\frac{\lambda }{c}\sin \left( \frac{u}{c}\right) ,-\phi ^{\prime
}(u)\cos v,-\phi ^{\prime }(u)\sin v\right) ,
\end{eqnarray*}%
where 
\begin{equation*}
\left \Vert X_{u}\times X_{v}\right \Vert =\sqrt{g_{11}g_{22}-g_{12}^{2}}%
=\lambda \cos \left( \frac{u}{c}\right) \neq 0.
\end{equation*}

The second partial derivatives of $X(u,v)$ are expressed as follows%
\begin{eqnarray*}
X_{uu}(u,v) &=&(\phi {}^{\prime \prime }(u),\frac{-\lambda }{c^{2}}\cos
\left( \frac{u}{c}\right) \cos v,\frac{-\lambda }{c^{2}}\cos \left( \frac{u}{%
c}\right) \sin v), \\
X_{uv}(u,v) &=&(0,\frac{\lambda }{c}\sin \left( \frac{u}{c}\right) \sin v,-%
\frac{\lambda }{c}\sin \left( \frac{u}{c}\right) \cos (v)), \\
X_{vv}(u,v) &=&(0,-\lambda \cos \left( \frac{u}{c}\right) \cos v,-\lambda
\cos \left( \frac{u}{c}\right) \sin (v)).
\end{eqnarray*}

Similarly, the coefficients of the second fundamental form of the surface
are 
\begin{eqnarray}
L_{11} &=&\text{ }<X_{uu}(u,v),N(u,v)>\text{ }=-\kappa _{1}(u),  \notag \\
L_{12} &=&\text{ }<X_{uv}(u,v),N(u,v)>\text{ }=0,  \label{C4} \\
L_{22} &=&\text{ }<X_{vv}(u,v),N(u,v)>\text{ }=\phi ^{\prime }{}(u)\lambda
\cos \left( \frac{u}{c}\right)  \notag
\end{eqnarray}%
where 
\begin{equation}
\kappa _{1}(u)=-\frac{\lambda }{c^{2}}\phi ^{\prime }{}(u)\cos \left( \frac{u%
}{c}\right) +\frac{\lambda }{c}\phi {}^{\prime \prime }(u)\sin \left( \frac{u%
}{c}\right) ,  \label{C5}
\end{equation}%
is the differentiable function. Furthermore, substituting (\ref{C4}) into ( %
\ref{A6})-(\ref{A7}) we obtain the following result.

\begin{proposition}
Let $M$ be a spherical surface in $\mathbb{E}^{3}$ given with the
parametrization (\ref{C3}). Then the Gaussian and mean curvature of $M$
become 
\begin{equation*}
K=1/c^{2},
\end{equation*}
and%
\begin{equation*}
H=\frac{\frac{2\lambda ^{2}}{c^{2}}\cos ^{2}\left( \frac{u}{c}\right) -\frac{%
\lambda ^{2}}{c^{2}}+1}{2\lambda \cos \left( \frac{u}{c}\right) \sqrt{1-%
\frac{\lambda ^{2}}{c^{2}}\sin ^{2}\left( \frac{u}{c}\right) }},
\end{equation*}%
respectively.
\end{proposition}

\begin{corollary}
\cite{Ve} Let $M$ be a spherical surface in $\mathbb{E}^{3}$ given with the
parametrization (\ref{C3}). Then we have the following asertations

$i)$ If $\lambda =c$ then the corresponding surface is a sphere with radius $%
c$ and centered at the origin,

$ii)$ If \ $\lambda >c$ then the corresponding surface is a hyperbolic
spherical surface,

$iii)$ If \ $\lambda <c$ then the corresponding surface is an elliptic
spherical surface.
\end{corollary}

\textbf{CASE }$\mathbf{II}$\textbf{. } For $n=2$ and $m=2$, the radius
vector (\ref{C1}) satisfying the indicated properties describes the \textit{%
generalized spherical surface }given with the radius vector 
\begin{equation}
X(u,v)=(f_{1}(u),f_{2}(u),\lambda \cos \left( \frac{u}{c}\right) \cos
v,\lambda \cos \left( \frac{u}{c}\right) \sin v),  \label{c4}
\end{equation}%
where%
\begin{equation}
\begin{array}{c}
f_{1}(u)=\int \sqrt{1-\frac{\lambda ^{2}}{c^{2}}\sin ^{2}\left( \frac{u}{c}%
\right) }\cos \alpha (u)du, \\ 
f_{2}(u)=\int \sqrt{1-\frac{\lambda ^{2}}{c^{2}}\sin ^{2}\left( \frac{u}{c}%
\right) }\sin \alpha (u)du.%
\end{array}
\label{c5}
\end{equation}%
are differentiable functions.

We call this surface the \textit{generalized spherical surface of first kind}%
. Actually, these surfaces are the special type of rotational surfaces \cite%
{GM1}, see also \cite{BABO1}.

The tangent space is spanned by the vector fields%
\begin{eqnarray*}
X_{u}(u,v) &=&(f_{1}{}^{\prime }(u),f_{2}{}^{\prime }(u),\frac{-\lambda }{c}%
\sin \left( \frac{u}{c}\right) \cos v,\frac{-\lambda }{c}\sin \left( \frac{u%
}{c}\right) \sin v), \\
X_{v}(u,v) &=&(0,0,-\lambda \cos \left( \frac{u}{c}\right) \sin v,\lambda
\cos \left( \frac{u}{c}\right) \cos (v)).
\end{eqnarray*}%
\ 

Hence, the coefficients of the first fundamental form of the surface are%
\begin{eqnarray*}
g_{11} &=&\text{ }<X_{u}(u,v),X_{u}(u,v)>\text{ }=1 \\
g_{12} &=&\text{ }<X_{u}(u,v),X_{v}(u,v)>\text{ }=0 \\
g_{22} &=&\text{ }<X_{v}(u,v),X_{v}(u,v)>\text{ }=\lambda ^{2}\cos
^{2}\left( \frac{u}{c}\right) ,
\end{eqnarray*}%
where $\left \langle ,\right \rangle $ is the standard scalar product in $%
\mathbb{E}^{4}$.

The second partial derivatives of $X(u,v)$ are expressed as follows%
\begin{eqnarray*}
X_{uu}(u,v) &=&(f_{1}{}^{\prime \prime }(u),f_{2}{}^{\prime \prime }(u),%
\frac{-\lambda }{c^{2}}\cos \left( \frac{u}{c}\right) \cos v,\frac{-\lambda 
}{c^{2}}\cos \left( \frac{u}{c}\right) \sin v), \\
X_{uv}(u,v) &=&(0,0,\frac{\lambda }{c}\sin \left( \frac{u}{c}\right) \sin v,-%
\frac{\lambda }{c}\sin \left( \frac{u}{c}\right) \cos (v)), \\
X_{vv}(u,v) &=&(0,0,-\lambda \cos \left( \frac{u}{c}\right) \cos v,-\lambda
\cos \left( \frac{u}{c}\right) \sin (v)).
\end{eqnarray*}%
The normal space is spanned by the vector fields%
\begin{eqnarray*}
N_{1} &=&\frac{1}{\kappa }(f_{1}{}^{\prime \prime }(u),f_{2}{}^{\prime
\prime }(u),\frac{-\lambda }{c^{2}}\cos \left( \frac{u}{c}\right) \cos v,%
\frac{-\lambda }{c^{2}}\cos \left( \frac{u}{c}\right) \sin v) \\
N_{2} &=&\frac{1}{\kappa }(\frac{\text{-}\lambda f_{2}{}^{\prime }(u)}{c^{2}}%
\cos \left( \frac{u}{c}\right) +\frac{\lambda f_{2}{}^{\prime \prime }(u)}{c}%
\sin \left( \frac{u}{c}\right) ,\frac{\text{-}\lambda f_{1}{}^{\prime \prime
}(u)}{c}\sin \left( \frac{u}{c}\right) +\frac{\lambda f_{1}{}^{\prime }(u)}{%
c^{2}}\cos \left( \frac{u}{c}\right) , \\
&&(f_{1}{}^{\prime }(u)f_{2}{}^{\prime \prime }(u)-f_{1}{}^{\prime \prime
}(u)f_{2}{}^{\prime }(u))\cos v,(f_{1}{}^{\prime }(u)f_{2}{}^{\prime \prime
}(u)-f_{1}{}^{\prime \prime }(u)f_{2}{}^{\prime }(u))\sin v)
\end{eqnarray*}%
where 
\begin{equation}
\kappa =\sqrt{(f_{1}{}^{\prime \prime })^{2}+(f_{2}{}^{\prime \prime })^{2}+%
\frac{\lambda ^{2}}{c^{4}}\cos ^{2}\left( \frac{u}{c}\right) },  \label{c10}
\end{equation}%
is the curvature of the profile curve $\gamma$. Hence, the coefficients of
the second fundamental form of the surface are%
\begin{eqnarray}
L_{11}^{1} &=&<X_{uu}(u,v),N_{1}(u,v)>=\kappa (u),  \notag \\
L_{12}^{1} &=&<X_{uv}(u,v),N_{1}(u,v)>=0,  \notag \\
L_{22}^{1} &=&<X_{vv}(u,v),N_{1}(u,v)>=\frac{\lambda ^{2}\cos ^{2}\left( 
\frac{u}{c}\right) }{c^{2}\kappa (u)},  \label{C11} \\
L_{11}^{2} &=&<X_{u}(u,v),N_{2}(u,v)>=0,  \notag \\
L_{12}^{2} &=&<X_{uv}(u,v),N_{2}(u,v)>=0,  \notag \\
L_{22}^{2} &=&<X_{vv}(u,v),N_{2}(u,v)>=-\frac{\lambda \cos \left( \frac{u}{c}%
\right) \kappa _{1}(u)}{\kappa (u)}.  \notag
\end{eqnarray}%
where 
\begin{equation}
\kappa _{1}(u)=f_{1}{}^{\prime }(u)f_{2}{}^{\prime \prime
}(u)-f_{1}{}^{\prime \prime }(u)f_{2}{}^{\prime }(u),  \label{C12}
\end{equation}%
is the differentiable function.

Furthermore, by the use of (\ref{C11}) with (\ref{A6})-(\ref{A7}) we obtain
the following results.

\begin{proposition}
The generalized spherical surface of first kind has constant Gaussian
curvature $K=1/c^{2}$.
\end{proposition}

\begin{proposition}
Let $M$ be a generalized spherical surface of first kind given with the
surface patch \textit{(\ref{c4})}. Then the mean curvature vector of $M$
becomes%
\begin{equation}
\overrightarrow{H}=\frac{1}{2}\left \{ \left( \frac{\kappa ^{2}c^{2}+1}{%
c^{2}\kappa }\right) N_{1}-\frac{\kappa _{1}}{\kappa \lambda \cos \left( 
\frac{u}{c}\right) }N_{2}\right \} .  \label{c15}
\end{equation}%
where%
\begin{equation}
\kappa =\sqrt{(\varphi \prime )^{2}+\varphi ^{2}\left( (\alpha ^{\prime
})^{2}+\frac{1}{c^{2}}\right) +\frac{\lambda ^{2}}{c^{4}}\left( 1-\frac{c^{2}%
}{\lambda ^{2}}\right) },\text{ }\kappa _{1}=\varphi ^{2}\alpha ^{\prime },
\label{c16}
\end{equation}%
and%
\begin{equation}
\varphi =\sqrt{1-\frac{\lambda ^{2}}{c^{2}}\sin ^{2}\left( \frac{u}{c}%
\right) }.  \label{c17}
\end{equation}
\end{proposition}

\begin{corollary}
Let $M$ be a generalized spherical surface of first kind given with the
surface patch \textit{(\ref{c4})}. If \ the second mean curvature $H_{2}$
vanishes identically then the angle function $\alpha (u)$ is a real constant.
\end{corollary}

For any local surface $M\subset \mathbb{E}^{4}$ given with the regular
surface patch $X(u,v)$ the normal curvature $K_{N}$ is given with the
following result.

\begin{proposition}
\cite{Be} Let $M\subset \mathbb{E}^{4}$ be a local surface given with a
regular patch $X(u,v)$ then the normal curvature $K_{N}$ of the surface
becomes%
\begin{equation}
K_{N}=\frac{%
g_{11}(L_{12}^{1}L_{22}^{2}-L_{12}^{2}L_{22}^{1})-g_{12}(L_{11}^{1}L_{22}^{2}-L_{11}^{2}L_{22}^{1})+g_{22}(L_{11}^{1}L_{12}^{2}-L_{11}^{2}L_{12}^{1})%
}{W^{3}}.  \label{C18}
\end{equation}
\end{proposition}

As a consequence of \textit{(\ref{C11})} with \textit{(\ref{C18})} we get
the following result.

\begin{corollary}
Any generalized spherical surface of first kind has flat normal connection,
i.e., $K_{N}=0$.
\end{corollary}

\begin{example}
In 1966, T. Otsuki considered the following special cases 
\begin{eqnarray*}
a)\text{ \ }f_{1}(u) &=&\frac{4}{3}\cos ^{3}(\frac{u}{2}),\text{ }f_{2}(u)=%
\frac{4}{3}\sin ^{3}(\frac{u}{2}),\text{ }f_{3}(u)=\sin u, \\
b)\text{ \ }f_{1}(u) &=&\frac{1}{2}\sin ^{2}u\cos (2u),\text{ }f_{2}(u)=%
\frac{1}{2}\sin ^{2}u\sin (2u),\text{ }f_{3}(u)=\sin u.
\end{eqnarray*}%
For the case a) the surface is called Otsuki (non-round) sphere in $\mathbb{E%
}^{4}$ which does not lie in a 3-dimensional subspace of $\mathbb{E}^{4}$.
It has been shown that these surfaces have constant Gaussian curvature \cite%
{Ot}.
\end{example}

\textbf{CASE }$\mathbf{III}$\textbf{. } For $n=1$ and $m=3$, the radius
vector \textit{(\ref{C1})} satisfying the indicated properties describes the 
\textit{generalized spherical surface }given with the radius vector%
\begin{equation}
X(u,v)=\phi (u)\overrightarrow{e_{1}}+\lambda \cos \left( \frac{u}{c}\right)
\rho (v),  \label{C19}
\end{equation}%
where%
\begin{equation}
\phi (u)=\int \sqrt{1-\frac{\lambda ^{2}}{c^{2}}\sin ^{2}\left( \frac{u}{c}%
\right) }du,  \label{C20}
\end{equation}%
and $\rho =\rho (v)$ parametrized by%
\begin{equation*}
\rho (v)=(g_{1}(v),g_{2}(v),g_{3}(v)),
\end{equation*}%
\begin{equation*}
\left \Vert \rho (v)\right \Vert =1,\left \Vert \rho ^{\prime }(v)\right
\Vert =1,
\end{equation*}%
which lies on the unit sphere $S^{2}\subset $ $\mathbb{E}^{4}.$ The
spherical curve $\rho $ has the following Frenet Frames;%
\begin{eqnarray*}
\rho ^{\prime }(v) &=&T(v) \\
T^{\prime }(v) &=&\kappa _{\rho }(v)N(v)-\rho (v) \\
N^{\prime }(v) &=&-\kappa _{\rho }(v)T(v).
\end{eqnarray*}

We call this surface a \textit{generalized spherical surface of second kind}%
. Actually, these surfaces are the special type of meridian surface defined
in \cite{GM2}, see also \cite{ABM}.

\begin{proposition}
Let $M$ be a meridian surface in $\mathbb{E}^{4}$ given with the
parametrization \textit{(\ref{C19})}. Then $M$ has the Gaussian curvature 
\begin{equation}
K=-\frac{\kappa _{\gamma }\phi ^{\prime }(u)}{\lambda \cos \left( \frac{u}{c}%
\right) {}},  \label{C21}
\end{equation}%
where%
\begin{equation*}
\kappa _{\gamma }(u)=-\frac{\lambda }{c^{2}}\phi ^{\prime }(u)\cos \left( 
\frac{u}{c}\right) +\frac{\lambda }{c}\phi ^{\prime \prime }(u)\sin \left( 
\frac{u}{c}\right) .
\end{equation*}
\end{proposition}

\begin{proof}
Let $M$ be a meridian surface in $\mathbb{E}^{4}$ defined by \textit{(\ref%
{C19})}. Differentiating \textit{(\ref{C19})} with respect to $u$ and $v$
and we obtain

\begin{eqnarray}
X_{u} &=&\phi ^{\prime }(u)\overrightarrow{e}_{1}-\frac{\lambda }{c}\sin
\left( \frac{u}{c}\right) \rho (v),\text{ \ }  \notag \\
X_{v} &=&\lambda \cos \left( \frac{u}{c}\right) \rho ^{\prime }(v),\text{ } 
\notag \\
X_{uu} &=&\phi ^{\prime \prime }(u)\overrightarrow{e}_{1}-\frac{\lambda }{%
c^{2}}\cos \left( \frac{u}{c}\right) \rho (v),\text{ \ }  \label{C22} \\
X_{uv} &=&-\frac{\lambda }{c}\sin \left( \frac{u}{c}\right) \rho ^{\prime
}(v),\text{ \ }  \notag \\
X_{vv} &=&\lambda \cos \left( \frac{u}{c}\right) \rho ^{\prime \prime }(v).%
\text{ \ }  \notag
\end{eqnarray}

The normal space of $M$ is spanned by

\begin{eqnarray}
N_{1} &=&N(v),  \label{C23} \\
N_{2} &=&-\frac{\lambda }{c}\sin \left( \frac{u}{c}\right) \overrightarrow{e}%
_{1}-\phi ^{\prime }(u)\rho (v),  \notag
\end{eqnarray}%
where $N(v)$ is the normal vector of the spherical curve $\rho$.

Hence, the coefficients of first and second fundamental forms are becomes%
\begin{eqnarray}
g_{11} &=&\text{ }<X_{u}(u,u),X_{u}(u,u)>\text{ }=1,\text{ }  \notag \\
g_{12} &=&\text{ }<X_{u}(u,v),X_{v}(u,v)>\text{ }=0,\text{ }  \label{C24} \\
g_{22} &=&\text{ }<X_{v}(v,v),X_{v}(v,v)>\text{ }=\lambda ^{2}\cos
^{2}\left( \frac{u}{c}\right) ,  \notag
\end{eqnarray}%
and%
\begin{eqnarray}
L_{11}^{1} &=&L_{12}^{1}=L_{12}^{2}=0,\text{ \ }  \notag \\
L_{22}^{1} &=&\kappa _{\rho }(v)\lambda \cos \left( \frac{u}{c}\right) ,%
\text{ }  \label{C25} \\
L_{11}^{2} &=&-\kappa _{\gamma }(u),\text{ \ }  \notag \\
L_{11}^{2} &=&\phi ^{\prime }(u)\lambda \cos \left( \frac{u}{c}\right) .%
\text{ \  \ }  \notag
\end{eqnarray}%
respectively, where%
\begin{eqnarray*}
\kappa _{\gamma }(u) &=&f_{1}^{\prime }(u)f_{2}^{\prime \prime
}(u)-f_{1}^{\prime \prime }(u)f_{2}^{\prime }(u) \\
&=&-\frac{\lambda }{c^{2}}\phi ^{\prime }(u)\cos \left( \frac{u}{c}\right) +%
\frac{\lambda }{c}\phi ^{\prime \prime }(u)\sin \left( \frac{u}{c}\right) .
\end{eqnarray*}

Consequently, substituting \textit{(\ref{C24})}-\textit{(\ref{C25})} into 
\textit{(\ref{A6})} we obtain the result.
\end{proof}

As a consequence of \textit{(\ref{C25})} with \textit{(\ref{C18})} we get
the following result.

\begin{proposition}
Any generalized spherical surface of second kind has flat normal connection,
i.e., $K_{N}=0$.
\end{proposition}

\begin{corollary}
Every generalized spherical surface of second kind is a meridian surface
given with the parametrization 
\begin{equation}
\begin{array}{c}
f_{1}(u)=\int \sqrt{1-\frac{\lambda ^{2}}{c^{2}}\sin ^{2}\left( \frac{u}{c}%
\right) }du, \\ 
\text{ \  \  \  \  \  \  \  \  \ }f_{2}(u)=\lambda \cos \left( \frac{u}{c}\right) .%
\text{ \  \  \  \  \  \  \  \  \  \  \  \  \  \  \  \  \  \  \  \  \  \  \  \  \  \  \  \  \ }%
\end{array}
\label{C26}
\end{equation}
\end{corollary}

By the use of \  \textit{(\ref{C20})-(\ref{C21}) with (\ref{C26}) we get the
following result.}

\begin{corollary}
The generalized spherical surface of second kind has constant Gaussian
curvature $K=1/c^{2}$.
\end{corollary}

As consequence of \textit{(\ref{A7})} we obtain the following result.

\begin{proposition}
Let $M$ be a generalized spherical surface of second kind given with the
parametrization (\ref{C19}). Then the mean curvature vector of $M$ becomes%
\begin{equation}
\overrightarrow{H}=\frac{1}{2f_{2}(u)}\left \{ \kappa _{\rho
}(v)N_{1}+\left( -\kappa _{\gamma }f_{2}(u)+f_{1}^{\prime }(u)\right)
N_{2}\right \} .  \label{C27}
\end{equation}%
where%
\begin{equation*}
\kappa _{\rho }(v)=\sqrt{g_{1}^{\prime \prime }(v)^{2}+g_{2}^{\prime \prime
}(v)^{2}+g_{3}^{\prime \prime }(v)^{2}}.
\end{equation*}
\end{proposition}

\begin{corollary}
Let $M$ be a generalized spherical surface of second kind given with the
parametrization (\ref{C19}). If 
\begin{equation}
\kappa _{\gamma }(u)=\frac{f_{1}^{\prime }(u)}{f_{2}(u)},  \label{C28}
\end{equation}%
then $M$ has vanishing second mean curvature, i.e., $H_{2}=0.$
\end{corollary}

\begin{example}
Consider the curve $\rho (v)=\left( \cos v,\cos v\sin v,\sin ^{2}v\right) $
in $S^{2}\subset \mathbb{E}^{3}$. The corresponding \textit{generalized
spherical surface }%
\begin{eqnarray}
x_{1}(u,v) &=&\int \sqrt{1-\frac{\lambda ^{2}}{c^{2}}\sin ^{2}\left( \frac{u%
}{c}\right) }du  \notag \\
x_{2}(u,v) &=&\lambda \cos \left( \frac{u}{c}\right) \cos v  \label{C29} \\
x_{3}(u,v) &=&\lambda \cos \left( \frac{u}{c}\right) \cos v\sin v  \notag \\
x_{4}(u,v) &=&\lambda \cos \left( \frac{u}{c}\right) \sin ^{2}v.  \notag
\end{eqnarray}%
is of second kind.
\end{example}

\begin{tabular}{l}
Beng\"{u} K\i l\i \c{c} Bayram \\ 
Department of Mathematics \\ 
Bal\i kesir University \\ 
Bal\i kesir, TURKEY \\ 
E-mail: benguk@bal\i kesir.edu.tr%
\end{tabular}

\bigskip

\begin{tabular}{l}
Kadri Arslan and Bet\"{u}l Bulca \\ 
Department of Mathematics \\ 
Uluda\u{g} University \\ 
16059 Bursa, TURKEY \\ 
E-mail: arslan@uludag.edu.tr,\ bbulca@uludag.edu.tr\  \  \  \ 
\end{tabular}

\bigskip

\end{document}